\documentclass[psamsfont]{amsart}
\usepackage[utf8]{inputenc}
\usepackage{graphicx}
\usepackage[dvips]{epsfig}
\usepackage{pinlabel}
\usepackage{amsmath}
\usepackage{amsfonts}
\usepackage{latexsym}
\usepackage{amssymb}
\usepackage[usenames,dvipsnames]{color}
\usepackage{amsthm}
\usepackage[all]{xypic}
\usepackage{enumitem}
\usepackage[breaklinks=true]{hyperref}

\input xy

\xyoption{all}

\author{Baptiste Chantraine}
\author{Emmy Murphy}
\address{Universit\'e de Nantes, France.}
\email{baptiste.chantraine@univ-nantes.fr}
\address{Northwestern University, United States.}
\email{e\_murphy@math.northwestern.edu}

\theoremstyle{plain}
\newtheorem{thm}{Theorem}[section]

\newtheorem{prop}[thm]{Proposition}
\newtheorem{coro}[thm]{Corollary}

\theoremstyle{definition}

\newtheorem{definition}[thm]{Definition}
\newtheorem{remark}[thm]{Remark}

\newtheorem{ex}[thm]{Example}
\theoremstyle{remark}
\newtheorem*{note}{Note}
\numberwithin{equation}{section}

\newcommand{\R}{\mathbb{R}}
\newcommand{\Z}{\mathbb{Z}}
\newcommand{\C}{\mathbb{C}}
\newcommand{\F}{\mathbb{F}}

\newcommand{\LL}{\mathcal{L}}

\newcommand{\p}{\varphi}
\newcommand{\e}{\varepsilon}
\newcommand{\dd}{\partial}
\newcommand{\ip}{\, \lrcorner \,}
\newcommand{\es}{\varnothing}
\newcommand{\op}{\operatorname}
\newcommand{\sse}{\subseteq}

\newcommand{\x}{\times}
\newcommand{\sm}{\setminus}
\newcommand{\pt}{\text{point}}

\newcommand{\wt}{\widetilde}

\newcommand{\ol}{\overline}

\newcommand{\std}{\text{std}}
\newcommand{\ot}{\text{ot}}
\newcommand{\conf}{\text{conf}}

\newcommand{\nov}{\text{Nov}}

\begin{document}
\title{Conformal symplectic geometry of cotangent bundles}
\thispagestyle{empty}
\maketitle

\begin{abstract}
We prove a version of the Arnol'd conjecture for Lagrangian submanifolds of conformal symplectic manifolds: a Lagrangian $L$ which has non-zero Morse-Novikov homology for the restriction of the Lee form $\beta$ cannot be disjoined from itself by a $C^0$-small Hamiltonian isotopy. Furthermore for generic such isotopies the number of intersection points equals at least the sum of the free Betti numbers of the Morse-Novikov homology of $\beta$. We also give a short exposition of conformal symplectic geometry, aimed at readers who are familiar with (standard) symplectic or contact geometry.
\end{abstract}

\section{Introduction}
\label{sec:introduction-1}

A \emph{conformal symplectic structure} on a manifold $M$ is a generalization of a symplectic structure. Locally, a conformal symplectic manifold is equivalent to a symplectic manifold, but the local symplectic structure is only well-defined up to scaling by a constant, and the monodromy of the local symplectic structure around curves may induce these rescalings. To our knowledge, the notion was first introduced by Vaisman in \cite{Vaisman_lck} and \cite{Vaisman}. They also appeared in work of Guedira and Lichnerowicz in \cite{Guedira_lichnerow}. It was later studied by Banyaga \cite{Banyaga_propertieslcs}, among many others.\footnote{In previous literature, this structure is often called a \emph{locally} conformal symplectic structure, which is a more accurate term. But also more cumbersome.} More recently a first general existence result was given in \cite{Apo_Dlou} for complex surfaces with odd first Betti number, it was later proved in \cite{E_M_Symp_Cob} that an almost symplectic manifold with non zero first Betti number admits a conformal symplectic structure, providing a large class of examples of such structures.
 
More formally, we can take a number of equivalent definitions. We can say that a conformal symplectic structure on $M$ is an atlas of charts to $\mathbb{R}^{2n}$, so that the transition maps $\psi_{ij}$ satisfy $\psi_{ij}^*\omega_\std = c_{ij}\omega_\std$, for constants $c_{ij} > 0$. Equivalently, let $E \to M$ be a flat, orientable, real line bundle. Then a conformal symplectic structure is a $2$-form on $M$ taking values in $E$, which is closed and non-degenerate.

Taking a connection on $E$ leads to the most tractable definition. A \emph{conformal symplectic structure} on $M$ is a pair $(\eta, \omega) \in \Omega^1M \x \Omega^2M$, so that $d\eta = 0$, $d\omega = \eta \wedge \omega$, and $\omega^{\wedge n} \neq 0$. Because the choice of connection is non-canonical, $(\eta, \omega)$ defines the same conformal symplectic structure as $(\eta + df, e^f\omega)$, for any $f \in C^\infty M$. (The choice of a specific $\eta$ representing $[\eta]$ is roughly analogous to a choice of contact form for a given contact structure, $\eta$ is called a \textit{Lee form} of the conformal symplectic structure since such a form appeared in the work of Lee \cite{Lee_evendim}.)

For an example of a conformal symplectic manifold, let $\beta$ be a closed $1$-form on a manifold $Q$, let $\lambda = p\cdot dq$ be the tautological $1$-form on $T^*Q$, let $\eta \in \Omega^1(T^*Q)$ be the pullback of $\beta$ under the projection, and let $\omega = d\lambda - \eta \wedge \lambda$.

Conformal symplectic manifolds enjoy many of the properties that make symplectic manifolds interesting. The definitions of Lagrangian, isotropic, etc. are exactly the same, and there is a natural notion of exact conformal symplectic structures, and exact Lagrangians inside them. They have a natural Hamiltonian dynamics (a smooth function defines a flow preserving the structure). They satisfy a Moser-type theorem, which implies that Darboux's theorem and the Weinstein tubular neighborhood theorems hold in this context (a small neighborhood of any Lagrangian is equivalent to the example above). When restricted to a coisotropic, the kernel of $\omega$ is a foliation, and in the case the leaf space is a manifold it inherits a conformal symplectic structure. The Poisson bracket on Hamiltonians intertwines the Lie bracket.

However, many of the more modern methods in symplectic geometry cannot be easily generalized, and often the theorems fail to hold true. For example, suppose that $\beta \in \Omega^1 Q$ is a $1$-form which never vanishes. Then in the conformal symplectic manifold $(T^*Q, \eta, \omega)$ defined above, the zero section $Z = \{p=0\}$ is displaceable by Hamiltonian isotopy. In fact, if $\p_t$ is the Hamiltonian isotopy generated by the Hamiltonian $H = 1$, then $\p_t(Z) \cap Z = \es$ for any $t>0$.

From the point of view of Floer theory, the problem with conformal symplectic structures is that we cannot even get started. The set of almost complex structures $J$ compatible with $\omega$ is still a non-empty contractible space, and $\ol\dd_J$ is still an elliptic operator, but because $\omega$ is not closed we have no bounds on energy, and therefore we do not expect Gromov compactness to hold, even for conformal symplectic manifolds which are both closed and exact. We discuss explicit examples suggesting failure of Gromov compactness in Section \ref{sec:failure-compactness}. Whether Gromov compactness can be generalized to this context by defining a more sophisticated compactification remains to be seen. 

The main theorem of this paper proves the following an analogous of Laudenbach-Sikorav's Theorem \cite{Lau_Sik} in the context of conformal symplectic manifolds. We let $H^\nov_*(L, [\eta]; \F)$ be the Novikov homology of $L$ in the homology class $[\eta] \in H^1(L; \R)$.

\begin{thm}\label{thm:rigidprinc}
Let $\eta$ be a closed $1$-form on an orientable manifold $Q$, and let $\F$ be a field. Let $\phi$ be an Hamiltonian diffeomorphism of the conformal symplectic manifold $(T^*Q,\eta,d\lambda-\eta\wedge\lambda)$ (where $\lambda$ is the canonical form) such that $\phi(Q_0)$ intersect $Q_0$ transvesally, then $\#\{\phi(Q_0) \cap Q_0\}\geq \op{rk} H^\nov_*(Q, [\eta]; \F)$. If $Q$ is non-orientable then $\#\{\phi(Q_0) \cap Q_0\}\geq \op{rk} H^\nov_*(Q, [\eta]; \Z_2)$
\end{thm}

\begin{coro}\label{cor:C0 close}
Let $(M, \eta, \omega)$ be a conformal symplectic manifold, and let $L \sse M$ be a Lagrangian. If $\p_t$ is any $C^0$ small Hamiltonian isotopy, then $\#\{\p_1(L) \cap L\} \geq \op{rk} H^\nov_*(L, [\eta]|_L; \Z_2)$. If $L$ is oriented we may replace $\Z_2$ with any field $\F$.
\end{coro}

In order to prove Theorem \ref{thm:rigidprinc} we prove an analogous to Sikorav's Theorem \cite{Sikorav_famgen} about persistence of generating families which allows to provide bounds for intersection of Lagrangian submanifolds in conformal contangent in terms of stable $\eta$-critical points of function. 

The layout of the paper follows. Section \ref{sec:introduction} introduces the basic definitions and theorems in conformal symplectic geometry. We include a number of propositions which are not necessary for the proof of Theorem \ref{thm:rigidprinc}, with the hope of giving a large-scale introduction to the theory, particularly for symplectic and contact geometers. 

Section \ref{sec:overv-morse-novik} gives a brief overview of Morse-Novikov homology needed for the proof of Theorem \ref{thm:rigidprinc}. Finally, in Section \ref{sec:rigid-lagr-inters} we complete the proof.

\section*{Acknowledgements}
\label{sec:aknowledgments}
The authors thank Vestislav Apostolov and François Laudenbach for inspiring discussions. The first author benefited from the hospitality of several institutions and wishes to thank the institute Mittag-Leffler in Stockholm, CIRGET in Montréal and MIT in Cambridge for the nice work environment they provided. The second author would like to thank Universit\'e de Nantes and the Radcliffe Institute for Advanced Study for their pleasant work environments. B. Chantraine is partially supported by the ANR project COSPIN (ANR-13-JS01-0008-01) and the ERC starting grant G\'eodycon. E. Murphy is partially supported by NSF grant DMS-1510305 and a Sloan Research Fellowship.

\section{Main definitions}
\label{sec:introduction}

\subsection{Conformal symplectic manifolds.}
\label{sec:conf-sympl-manif}

\begin{prop}\label{prop:equiv def}
Let $M$ be a $2n$-manifold. The following are equivalent:

\begin{itemize}
\item An atlas of charts $M = \bigcup U_i$, $\p_i:U_i \to \R^{2n}$, so that the transition maps $\psi_{ij} = \p_i \circ \p_j^{-1}$ preserve the standard symplectic form up to scaling by a positive local constant: $\psi_{ij}^*(\omega_\std) = c_{ij} \omega_\std$. Two atlases are considered equivalent if they admit a common refinement.

\item A flat, real, orientable line bundle $E \to M$, and a $2$-form $\sigma \in \Omega^2(M, E)$, so that $\sigma$ is non-degenerate (as a map $TM \to T^*M \otimes E$) and closed (as a form with values in a flat line bundle). Two such structures $(E_1, \sigma_1)$, $(E_2, \sigma_2)$ are considered equivalent if there is an isomorphism $\p: E_1 \to E_2$ covering the identity map, so that $\p^*\sigma_2 = \sigma_1$.

\item A pair $(\eta, \omega) \in \Omega^1M \x \Omega^2M$, so that $d\eta = 0$, $d\omega = \eta \wedge \omega$, and $\omega^{\wedge n} \neq 0$. $(\eta, \omega)$ is equivalent to $(\eta + df, e^f\omega)$ for any $f \in C^\infty M$.
\end{itemize}

Any of these structures is called a \emph{conformal symplectic structure} on $M$.
\end{prop}

\begin{proof}
Given an atlas, the association of the positive number $c_{ij}$ to every intersection $U_i \cap U_j$ can be thought of as the clutching function for a principal $GL^+(1, \R)$-bundle on $M$, with the discrete topology. That is, the numbers $c_{ij}$ define an orientable flat line bundle, $E$, and the form $\sigma =\p_i^*\omega_\std$ is well defined as a $2$-form taking values in $E$. It is closed and non-degenerate, since these are local conditions. 

Given a pair $(E, \sigma)$, we can identify $E \cong M \x \R$ as smooth vector bundles globally, since $E$ is a real, orientable line bundle. A choice of flat connection on $E$ then is simply a closed $1$-form $\eta \in \Omega^1(M)$. In this way we can identify $\sigma$ as an ordinary $2$-form $\omega \in \Omega^2(M)$. Then to say that $\omega$ is closed as a $2$-form with values in $E$ is equivalent to the statement $d\omega - \eta \wedge \omega = 0$. The choice of connection $\eta$ is not canonical; gauge symmetries of $E$ act by $\eta \mapsto \eta + df$ for any $f \in C^\infty(M)$, and this gauge symmetry acts on $\omega$ by $\omega \mapsto e^f\omega$.

Given a pair $(\eta, \omega)$ as above, consider the covering space $\pi: \wt M \to M$ associated to $[\eta]$. Then $\pi^*\eta = d\theta$ for some $\theta \in C^\infty \wt M$, and for any covering transformation $g \in \pi_1M / \ker [\eta] \sse \op{Diff}\wt M$, we have $\theta \circ g = \theta + \langle[\eta], g \rangle$. Let $\wt\omega = e^{-\theta}\pi^*\omega \in \Omega^2\wt M$. Then $\wt\omega$ is a symplectic form, and $g^*\wt\omega = e^{-\langle[\eta], g \rangle}\wt\omega$. Therefore by taking a Darboux atlas on $(\wt M, \wt\omega)$, we get a conformal symplectic atlas on $M$.
\end{proof}

\begin{definition}\label{def:deta}
Let $\eta$ be a closed $1$-form on a manifold $M$. The \emph{Lichnerowicz-De Rham differential} on $\beta \in \Omega^*(M)$ is defined as $d_\eta \beta=d\beta-\eta\wedge\beta$. 
\end{definition}

From the fact that $\eta$ is closed and of odd degree we get that $d_\eta^2=0$ and from the fact that $\eta$ has degree $1$ we see that $d_\eta$ is degree $1$ as well. Perhaps the most essential difference between $d_\eta$ and $d$ is that $d_\eta$ does not satisfy a Stokes' theorem. We note two important formulas:

$$d_{\eta + df}\beta = e^f d_\eta(e^{-f}\beta)$$
$$\LL_X\beta = X \ip d_\eta\beta + d_\eta(X \ip \beta) + \eta(X)\beta$$

where $\LL_X$ is the Lie derivative along the vector field $X$. By passing to the covering space of $M$ defined by $\ker [\eta]$, it follows that the homology $H^*(\Omega^*M, d_\eta)$ is isomorphic to $H^*(M; [\eta], \R)$, the homology of $M$ with local coefficients defined by the homomorphism $[\eta]: \pi_1M \to \R$. 

\begin{note}
For concreteness, we will state definitions and prove propositions in the setup where a conformal symplectic structure is a pair $(\eta, \omega)$, and then when relevant show that the definitions are invariant under gauge equivalence $\eta \rightsquigarrow \eta + df$. By working directly with the complex $\Omega^*(M, E)$ for the flat line bundle $E$, the arguments are more elegant, but also they are likely more opaque.
\end{note}

\begin{definition}\label{def:basics}
Let $(M, \eta, \omega)$ be a conformal symplectic structure. We say that a submanifold $L \sse M$ is \emph{isotropic} if $\omega|_L = 0$, \emph{coisotropic} if $TL^{\perp \omega} \sse TL$, and \emph{Lagrangian} if it is isotropic and coisotropic. 

We say that $(\eta, \omega)$ is \emph{exact} if $\omega = d_\eta\lambda$ for some $\lambda \in \Omega^1M$. In this case we say that $\lambda$ is a \emph{Liouville form} for $(\eta, \omega)$, and the vector field $Z_\lambda$ defined by $Z_\lambda \ip \omega = \lambda$ is called the \emph{Liouville vector field}.\footnote{This usage is somewhat different from the standard symplectic case, where a Liouville form is required to be a primitive of $\omega$ and also satisfy a convexity condition near the boundary/infinite portion of $M$. The same conditions on $Z_\lambda$ make sense in this context, so they could easily be imposed.} If $L \sse M$ is a Lagrangian in $(M, \eta, d_\eta\lambda)$, we say that $L$ is \emph{exact} if $\lambda|_L = d_\eta h$ for some $h \in C^\infty L$.
\end{definition}

Notice that all of the above definitions are well defined up to gauge equivalence, because if we replace $(\eta, \omega)$ with $(\eta + df, e^f\omega)$, we can also replace $\lambda$ and $h$ with $e^f\lambda$ and $e^fh$. In particular, the Liouville vector field $Z_\lambda$ is well defined independent of gauge, it is \emph{not} rescaled by $e^f$. Note also that there is nothing precluding the existence of a closed, exact conformal symplectic manifold.

\begin{ex}\label{ex:confsymplectisation}
Let $(Y,\lambda)$ be a manifold with a $1$-form $\lambda$. We denote by $S^1_T$ the quotient $\mathbb{R}/T\mathbb{Z}$ and parametrise it with the coordinate $\theta$. On $S_T^1\times Y$ let $\eta = -d\theta$. Then $\omega =d_\eta\lambda =d\lambda + d\theta\wedge\lambda$, so $\omega$ is non-degenerate if and only if $\lambda$ is a contact form. Furthermore, given another contact form defining the same cooriented contact structure, $e^f\lambda$, we have that the conformal Liouville manifold $(S_T^1 \times Y, \eta, \lambda)$ is gauge equivalent to $(S_T^1 \times Y, \eta+df, e^f\lambda)$, which is conformal symplectomorphic to $(S_T^1\times Y, \eta, e^f\lambda)$ under the coordinate change $(\theta, y) \mapsto (\theta - f(y), y)$.

The minimal cover which makes $\eta$ exact is $\mathbb{R}\times Y\rightarrow S^1\times Y$, and the conformal Liouville structure $(\eta, \lambda)$ pulls back to $(-dt, \lambda)$, which is gauge equivalent to the exact symplectic structure $e^t\lambda$, known as the symplectization of $(Y,\ker\lambda)$. Hence, the conformal symplectic manifold will be called the \textit{conformal symplectization} of $(Y,\xi)$. We denote this manifold by $S^\conf_T(Y, \ker\lambda)$. (Notice that $\langle[\eta], S^1 \x \{\pt\}\rangle = -T$, so the choice of $T$ affect the conformal symplectomorphism type.) If the choice of $T$ is irrelevant, we will use the notation $S^\conf(Y, \ker\lambda)$.

Now, let $\Lambda \sse Y$ be a Legendrian submanifold. Then the cylinder $\mathbb{R}\times \Lambda$ is an exact Lagrangian in the symplectization of $(Y, \ker \lambda)$, and furthermore it is invariant under the covering transformations of $\mathbb{R}\times Y\rightarrow S^1\times Y$. Therefore it descends to an exact Lagrangian submanifold $S^1 \x \Lambda$ of $S^\conf(Y, \ker\lambda)$. More generally, if $\Sigma$ is a Lagrangian cobordisms from $\Lambda$ to itself which is cylindrical outside of $(0,T)\times Y$ then $\Sigma$ descends to a Lagrangian submanifold of $S^\conf_T(Y, \ker\lambda)$.
\end{ex}

This paper will focus primarily on the following example:

\begin{ex}\label{ex:cotangent}
Let $\beta$ be a closed $1$-form on a smooth manifold $Q$, let $\pi$ be the projection $T^*Q\rightarrow Q$, and let $\lambda_\std$ be the tautological form on $T^*Q$. Define $\eta:=\pi^*\beta$. Then $(T^*Q,\eta, \lambda_\std)$ is a conformal Liouville manifold (since $\eta\wedge \lambda_\std \wedge d\lambda_\std^{n-1}=0$). We will denote this conformal Liouville manifold by $T^*_\beta Q$.

Let $\alpha:Q \rightarrow T^*Q$ be a $1$-form. Then $\alpha^*\lambda_\std=\alpha$ and thus $\alpha(Q)$ is Lagrangian in $T_\beta^*Q$ if and only if $d_\beta \alpha = 0$. Furthermore this Lagrangian is exact if and only if $\alpha$ is $d_\beta$-exact.
\end{ex}

\subsection{Conformal symplectic transformations.}
\label{ssec:conf-sympl-transf}

A diffeomorphism between conformal symplectic manifolds $\p: (M_1, \eta_1, \omega_1) \to (M_2, \eta_2, \omega_2)$ is called a \emph{conformal symplectomorphism} if $\p^*\eta_2 = \eta_1 + df$ and $\p^*\omega_2 = e^f \omega_1$ for some $f \in C^\infty M_1$. We denote by $\operatorname{Symp}(M,\eta,\omega)$ the group of conformal symplectomorphisms of $(M,\eta,\omega)$. The Lie algebra $\mathfrak{S}\operatorname{ymp}(M,\eta,\omega)$ of this group is given by vector fields $X$ satisfying $\LL_X\omega=f\omega$ and $\LL_X\eta=df$ for some $f$, and using Cartan's formula we see this is equivalent to
\begin{align*}
d_\eta(X \ip \omega) + \eta(X)\omega &= \LL_X \omega = f \omega\\
d(\eta(X)) &= \LL_X\eta =df
\end{align*}

Therefore $X$ is a conformal symplectic vector field if and only if $d_\eta(X\ip \omega) = c\omega$ for a constant $c \in \R$. 

The subalgebra of $\mathfrak{S}\operatorname{ymp}(M,\eta,\omega)$ which are $\omega$-dual to $\eta$-closed $1$-forms (i.e. vector fields $X$ with $d_\eta(X\ip \omega) = 0$) will be called \textit{divergence free conformal symplectic vector fields}. Notice the following fact: a conformal symplectic manifold is exact if and only if it admits a conformal symplectic vector field which is not divergence-free. In this case, after choosing a Liouville form $\lambda$, we see that every conformal symplectic vector field is of the form $X + cZ_\lambda$, where $X$ is divergence-free.

\begin{definition}
The conformal symplectic vector fields which are dual to $\eta$-exact forms is called \emph{Hamiltonian vector fields}. For any function $H \in C^\infty M$, the Hamiltonian vector field $X_H$ associated to $H$ is defined by $X_H \ip \omega = - d_\eta H$. A conformal symplectomorphism of $M$ which is the time-$1$ flow of a path of Hamiltonian vector fields will be called a \textit{Hamiltonian diffeomorphism}.
\end{definition}

Notice that the association $H \rightsquigarrow X_H$ depends on $\eta$, but the algebra of Hamiltonian vector fields (and therefore the group of Hamiltonian diffeomorphisms) does not. Using $X_H^\eta$ to denote this dependence, we immediately have $X_H^\eta = X_{e^f H}^{\eta + df}$. To phrase the dependence in a gauge free way, let $E$ be the flat line bundle determined by $[\eta]$, then Hamiltonians are taken to be $H \in \Omega^0(M, E)$. Since $\omega \in \Omega^2(M, E)$, the equation $X_H \ip \omega = - d_E H$ defines $X_H$ unambiguously.\footnote{The reader experienced with contact geometry will be familiar with this situation: the contact vector field $X_H$ associated to a Hamiltonian $H \in C^\infty M$ depends on a choice of contact form, but the correct way to define the relationship without reference to a contact form is to take contact Hamiltonians $H \in \Omega^0(M, TM/\xi)$.}

\begin{definition}
Given a conformal symplectic structure with a chosen connection $\eta$, we define the \emph{Lee vector field} to be the Hamiltonian vector field generated by $H = 1$. We sometimes denote this vector field as $R_\eta = X_1$.
\end{definition}

\begin{ex}
Given a contact manifold $(Y, \ker \alpha)$, we defined the conformal symplectization in Example \ref{ex:confsymplectisation} by $(S^1 \x Y, \eta = -d\theta, \omega = d_\eta\alpha)$. In this case, the Lee vector field of $\eta$ is equal to the Reeb vector field of $\alpha$.
\end{ex}

Note that this exhibit a first difference with the symplectic case:  this flow has no fixed point for small time. This is an example of a more general phenomenon: given a conformal symplectic manifold $(M, \eta, \omega)$, if $\eta$ can be chosen to have no zeros, then $R_\eta$ is a Hamiltonian vector field with no zeros, and therefore we can construct many Hamiltonian diffeomorphisms with no fixed points. Another place where this condition is relevant is about Lagrangian displaceabality:

  \begin{ex}
    Given a cotangent bundle $T_\beta^*Q$ with Liouville structure given by a closed one form $\beta$ on $Q$ as in Example \ref{ex:cotangent}, the Liouville vector field $Z_\lambda$ is given by  the standard Liouville vector field on $T^*Q$. The Lee vector field $R_\eta$ corresponds to the symplectic vector field dual to $\eta$ (in the standard symplectic sense), i.e. its flow acts fiberwise, and translates the vertical fiber $T_q^*Q$ in the direction $\beta_q$. In particular, if $[\beta] \in H^1Q$ can be represented by a non-vanishing closed $1$-form, then there exists an autonomous Hamiltonian flow on $T^*_\beta Q$ so that $\p_H^t(Q) \cap Q = \es$ for \emph{any} $t>0$.
  \end{ex}

Given an exact conformal symplectic manifold $(M, \eta, d_\eta \lambda)$, there is some interplay between the vector fields $Z_\lambda$ and $R_\eta$. Always, we have $$\lambda(R_\eta) = \omega(Z_\lambda, R_\eta) = - \eta(Z_\lambda).$$ Furthermore, suppose that $X \in \ker d\lambda$. Since $\omega = d_\eta\lambda = d\lambda - \eta \wedge \lambda$, we get that $X \ip \omega = \lambda(X)\eta - \eta(X)\lambda$, and therefore $X = \lambda(X)R_\eta - \eta(X)Z_\lambda$ since $\omega$ is non-degenerate. It follows that, if $d\lambda$ has non-zero kernel, it must be equal to the span of $R_\eta$ and $Z_\lambda$. Plugging either $R_\eta$ or $Z_\lambda$ into the original equation, we see that $d\lambda$ has non-zero kernel if and only if $\eta(Z_\lambda) = -1$, if and only if $\lambda(R_\eta) = 1$. 

In the literature, exact conformal symplectic manifolds satisfying $\eta(Z_\lambda) = -1$ everywhere are often referred to as conformal symplectic manifolds \emph{of the first kind}. Notice that this condition is not gauge invariant, and therefore it can be difficult to tell when a given conformal symplectic manifold is gauge-equivalent to one satisfying this property. However, this condition has strong implications, as shown in \cite[Theorem A]{Bazzoni_conf}: if $M$ is compact and $\eta(Z_\lambda) = -1$ everywhere, and if $\{\eta = 0\}$ has at least one compact leaf, then $M$ is conformal symplectomorphic to the suspension of a strict contactomorphism of a contact manifold $Y$. 

\subsection{Moser's theorem and tubular neighborhoods}
\label{ssec:moser}
We have the following version of Moser's Theorem (first proved in \cite{Banyaga_propertieslcs}).

\begin{thm}\label{thm:Moser}
Let $(M, \eta)$ be a closed smooth manifold equipped with a closed $1$-form, and let $\{\omega_t\}_{t \in [0,1]}$, be a path of conformal symplectic structures in the same homology class, i.e. $\omega_t = \omega_0 + d_\eta\lambda_t$. Then there is an isotopy $\p_t: M \to M$ and a path of smooth functions $f_t \in C^\infty M$ so that $\p_t^*\eta = \eta + df_t$ and $\p_t^*\omega_t = e^{f_t}\omega_0$. That is, an exact homotopy of conformal symplectic structures is always induced by an isotopy (and gauge transformations)
\end{thm}

\begin{proof}
We find a vector field $X_t$ suitable for $X_t = \dot{\p_t} \circ \p_t^{-1}$, and then using compactness of $M$ integrate $X_t$. Differentiating the desired equations gives
$$\p_t^*\left(\LL_{X_t}\eta\right) = d\dot f_t$$
$$\p_t^*\left(\dot\omega_t + \LL_{X_t}\omega_t\right) = \dot f_t e^{f_t}\omega_0 = \dot f_t \p_t^*\omega_t.$$

Letting $\mu_t = \dot f_t \circ \p_t^{-1}$, we have
$$d(\eta(X_t)) = \LL_{X_t}\eta = d\mu_t$$
$$\dot\omega_t + d_\eta(X_t \ip \omega_t) + \eta(X_t)\omega_t = \mu_t\omega_t.$$

Defining $X_t$ by $X_t \ip \omega = \dot\lambda_t$, and taking $\mu_t = \eta(X_t)$, we see that this system has a solution.
\end{proof}

Just as in the symplectic case, the Moser theorem for conformal symplectic structures is used to prove a Weinstein neighborhood theorem for Lagrangians in conformal symplectic manifolds, showing that Example \ref{ex:cotangent} is a universal model. More precisely we have the following result, which appears in the recent paper \cite{Otiman_Stanciu_darboux}. It also follows from a more general version for coisotropic manifolds, appearing in \cite[Theorem 4.2]{VanLe_Oh_cositrop}.

\begin{thm}\label{thm:lag nbhd}
Let $(M,\eta,\omega)$ be a conformal symplectic manifold and $L\subset M$ a compact Lagrangian. Let $\beta=\eta|_L$. Then there exists a neighborhood $\mathcal{N} \sse M$ of $L$ so that $\mathcal{N}$ is conformally symplectomorphic to a neighborhood $U$ inside $(T^*_\beta L, \pi^*\beta, d_{\pi^*\beta}\lambda_\std)$, sending $L \sse \mathcal{N}$ to the zero section $\{p=0\} \sse U$.
\end{thm}

\begin{proof}
Letting $\mathcal{N}$ be a tubular neighborhood of $L$, we have a diffeomorphism $\psi: T^*L \to \mathcal{N}$ sending the zero section $Z \sse T^*L$ to $L\sse \mathcal{N}$, so that $\psi^*\omega|_{T_ZT^*L} = d_{\pi^*\beta}\lambda_\std|_{T_ZT^*L}$. $\psi^*\eta$ and $\pi^*\beta$ are homologous, so by taking a gauge transformation on $\mathcal{N}$ we can assume that they are equal. 

Since $\psi^*\omega_1|_Z = 0$, therefore $\psi^*\omega = d_{\pi^*\beta}\lambda_1$ is exact (the Lichnerowicz-De Rham complex is homotopy invariant). We can furthermore assume that $\lambda_1|_{T_ZT^*L} = 0$ by addition of a $d_{\pi^*\beta}$-closed $1$-form. Let $\lambda_t = t\lambda_1 + (1-t)\lambda_\std$, and $\omega_t = d_\eta\lambda_t$. By taking a smaller neighborhood if necessary, we may assume that $\omega_t$ is non-degenerate for all $t \in [0,1]$.

We then run the Moser method, as above. The vector field $X_t$ obtained is zero along $Z$, and therefore by shrinking $\mathcal{N}$ further if necessary, we get a conformal symplectomorphism from a neighborhood of $Z \sse T^*_\beta L$ to $\mathcal{N}$. 
\end{proof}

\subsection{Flexibility and rigidity}\label{ssec:formal}

Given a closed manifold $M^{2n}$, we would like to know when $M$ admits a conformal symplectic structure. If so, how many distinct structures are there up to isotopy. There are some basic obstructions/invariants. For ease of notation we focus on the exact case.

\begin{definition}
Let $M^{2n}$ be a closed manifold. An \emph{exact almost conformal symplectic structure} (or EACS structure) is a pair $(a, w)$, where $a \in H^1(M; \R)$, and $w \in \Omega^2M$ is a non-degenerate $2$-form.
\end{definition}

Any exact conformal symplectic structure $(\eta, d_\eta\lambda)$ defines an almost conformal symplectic structure with $a = [\eta]$, $w = d_\eta\lambda$. Therefore, for a manifold to admit an exact conformal symplectic structure, it must also admit an EACS structure. Similarly, for two conformal symplectic structures to be isotopic, they necessarily must be homotopic through EACS structures (where $w$ is homotoped but $a$ is fixed). The benefit of passing to EACS structures is that they are classified purely by algebraic topology.

We show that that the classification of exact conformal symplectic structures is strictly more subtle than their almost conformal counterpart.

\begin{prop}\label{prop:nonflexible}
There exists a manifold $M$ of any dimension and and $ [\eta] \neq 0 \in H^1M$, which admits two exact conformal symplectic structures $(\eta, d_\eta\lambda_1)$ and $(\eta, d_\eta\lambda_2)$ which are homotopic through non-degenerate $2$-forms, but they are not conformally symplectomorphic.
\end{prop}

\begin{proof}
We take $M = S^1 \x S^{2n-1}$. Let $\xi_\ot = \ker \alpha_\ot$ be an overtwisted contact structure (see \cite{BEM}) on $S^{2n-1}$, which is homotopic through almost contact structures to the standard contact structure $\xi_\std$. 

As the symplectization of $(S^{2n-1},\xi_\ot)$ is different form the one of $(S^{2n-1},\xi_\std)$ (actually the second cannot embed in the first as $(S^{2n-1},\xi_\std)$  is fillable) we know that those two structure are not equivalent.
 
Moser's theorem therefore implies that they are not homotopic through exact conformal symplectic structures. But they are homotopic through non-degenerate $2$-forms, since the original contact structures are homotopic through almost contact structures.
\end{proof}

\begin{remark}
  As shown in \cite{Murphy_closedlagsymplectisation}, when $n>2$ the symplectization $(\R \x S^{2n-1}, e^r\alpha_\ot)$ contains an exact Lagrangian, $L \sse \R \x S^{2n-1}$. Let $R > 0$ be a constant so that $L \sse (0, R) \x S^{2n-1}$. Then by identifying $S^1 = \R / R\Z$ $L$ embeds as an exact Lagrangian into the conformal symplectization $(M, \eta = -d\theta, \lambda_1 = \alpha_\ot)$.

However, there can be no nulhomologous exact Lagrangian in the conformal symplectization of the standard contact structure, $(S^1 \x S^{2n-1}, \eta, \lambda_2 = \alpha_\std)$. For if there was, it would lift to a compact exact Lagrangian in the universal cover, which is conformally symplectomorphic to $\C^n \sm \{0\}$, which then contradicts Gromov's theorem. This shows that those two conformal symplectic structure can be distinguished by their Lagrangian submanifolds. 
\end{remark}
For $2n = 2$ we know that conformal symplectic structures are equivalent to almost conformal symplectic structures, since all $2$-forms are exact when $[\eta] \neq 0$. 

On the question of existence, a nearly complete answer was established in \cite{E_M_Symp_Cob}:

\begin{thm}
Let $M$ be a closed manifold with an EACS structure $(a, w)$, where $a \neq 0$ is in the image $H^1(M; \Z) \to H^1(M, \R)$. Then for all sufficiently large $C \in [1, \infty)$, there is an exact conformal symplectic structure $(\eta, d_\eta\lambda)$, with $[\eta] = Ca$ and $d_\eta\lambda$ being homotopic to $w$ through non-degenerate $2$-forms.

In particular, given any closed manifold $M$ satisfying $H^1(M; \R) \neq 0$ and having a non-degenerate $2$-form, $M$ admits an exact conformal symplectic structure.
\end{thm}

This follows indeed from the existence result of symplectic structure on almost symplectic cobordisms by cutting along an hypersurface Poincaré dual to $[\eta]$ which is non-separating by hypothesis. Making then the cobordism symplectic between the same contact structure on the ends leads to a symplectic cobordisms whose ends can be identified to give a conformal symplectic manifold.

It is an open question whether we can in general construct conformal symplectic structures where $[\eta]$ is a prescribed $1$-form, either in the case where $[\eta]$ is small, or when $[\eta]: \pi_1M \to \R$ has non-discrete image. Note that in \cite{Apo_Dlou}, there is one example that shows that after fixing a complex structure J, the set of $[\eta]$ for which there exists a conformal symplectic structure compatible with $J$ consists of a single point. 

\begin{remark}
From this theorem, it is easy to construct non-exact conformal symplectic structures: given an exact symplectic structure $(\eta, d_\eta\lambda)$ and a class $b \in H^2(M; [\eta], \R)$, choose a $d_\eta$-closed $2$-form $B \in \Omega^2M$ representing $b$. Then for sufficiently large $C$, $\omega = Cd_\eta\lambda + B$ is a conformal symplectic structure.
\end{remark}

\begin{remark}
The same strategy also works for constructing conformal symplectic manifolds with convex contact boundary: any contact manifold which admits an almost complex filling admits a conformal symplectic filling. Indeed from such a formal filling the main theorem in \cite{E_M_Symp_Cob} allows us to construct a symplectic cobordisms from $(S^{2n-1},\xi_\ot)$ to $(S^{2n-1},\xi_\ot)\sqcup (Y,\xi)$, and gluing the two $(S^{2n-1},\xi_\ot)$ leads to a conformal symplectic filling of $(Y, \xi)$.

As stated there, the theorem only applies to cobordisms with tight convex boundary when $n > 2$. However, after an examination of the proof there it is clear that the result applies when $n=2$ as well, as long as a single component of the convex boundary is overtwisted.
\end{remark}

\subsection{Pseudo-holomorphic curves and a ``failure'' of Gromov compactness.}
\label{sec:failure-compactness}

Inspired from successes in symplectic and contact geometry, we might try to develop a theory of pseudo-holomorphic curves for conformal symplectic manifolds. Similar to the symplectic case, in a conformal symplectic manifold the space of compatible almost complex structures $J$ is contractible, and the equation $\ol{\dd}_J u = 0$ is elliptic. But when trying to prove a compactness theorem, we run into an immediate problem: the Lichnerowicz-De Rham complex does not satisfy Stokes' theorem, and without this we have no way to give a $C^0$ bound on holomorphic energy. We give here two examples of holomorphic curves arising naturally which suggest bad compactness properties.

Consider the conformal sympectisation of an overtwisted sphere $Y = S^{2n-1}_\ot$ with a generic contact form. When equipped with a cylindrical almost complex structure $J$, there is a holomorphic plane which is asymptotic to one of the Reeb orbits, as shown in \cite{Albers_Hofer}. In the conformal symplectization $S^\conf(Y)$ this gives a holomorphic plane, which looks something like a leaf of a Reeb type foliation. Though such curves are handled in the symplectization setting by SFT compactness \cite{Bourgeois_&_Compactness} and \cite{Abbas_book}, this sort of compactness strongly relies on having contact asymptotics and a cylindrical complex structure, which is not a natural notion in general conformal symplectic manifolds.

A second phenomena is given by the conformal symplectic filling of $S^3_{\ot}$ described in the previous section. The bishop family near the elliptic point of an overtwisted disk leads to a family of holomorphic curves which cannot be compactified in a standard way. We suspect that such a family breaks into curves involving half-plane such as those described before.

\subsection{Contactisation, reduction, and generating functions}
\label{ssec:cont-red-gen}

\begin{definition}\label{def:contactization}
If $(M,\eta,\lambda)$ is a conformal Liouville manifold, we define the \emph{contactisation} of $(M, \eta, \lambda)$ to be the manifold $M\times \mathbb{R}$ equipped with the contact structure given by the kernel of $\alpha=dz-z\eta-\lambda=d_\eta z-\lambda$ (note that $\alpha$ is a contact form if and only if $d_\eta \lambda$ is non-degenerate).
\end{definition}

Notice that an exact Lagrangian submanifold $L$ lifts to a Legendrian submanifold $\widetilde{L}= \{(q,f(q)); q \in L\}$ where $f$ is such that $d_\eta f=\lambda|_L$.

In the cotangent case, the space $\mathcal{J}^1_\beta(Q):=T^*_\beta Q\times \mathbb{R}$ with the contact form $d_\beta z-\lambda$ is called the \emph{$\beta$-jet bundle}, natural Legendrian submanifolds arise as $\beta$-jets of functions $f$, i.e. $j^1_\beta f(q) = (q,d_\beta f_q,f(q))$. Note that $\mathcal{J}^1_\beta(Q)$ is actually contactomorphic to $\mathcal{J}^1(Q)$ by the map $\p(q,p,z) = (q,p-z\beta_q,z)$. However, this contactomorphism is not compatible with the projection to $T^*Q$ and therefore does not identify the symplectic properties of $T^*Q$ with the conformal symplectic properties of $T^*_\beta Q$. Rather, it maps $\beta$-jets of functions to (regular) jets of functions.

The Reeb vector field of $\alpha$ is given by $R_\alpha=(1+\lambda(R_\eta))\partial_z-R_\eta$.

\begin{ex}
  Let $(M,\alpha)$ be a contact manifold. Then the contact form on $M\times S^1\times\mathbb{R}$ is $dz-zd\lambda-\alpha$ the Reeb vector field is just the original Reeb vector field seen as vector fields invariant by $S^1\times\mathbb{R}$.
  The contactization of a conformal symplectization is a contact manifold associated to $(M,\ker \alpha)$, this manifold is the product of the original contact manifold and $T^*S^1$ with the Liouville structure given by $pd\theta-dp$.
\end{ex}

Let $(M,\eta,\omega)$ be a conformal symplectic manifold, and let $N\subset M$ be a coisotropic submanifold (i.e. such that $TN^{\perp \omega}\subset TN$). Since $d\omega=\eta\wedge \omega$ it follows from the Fröbenius integrability theorem that the distribution $\ker(\omega|_N)$ is integrable. Note that for a vector field $X$ in $\ker \omega|_N$ we have that
\begin{align}\label{eq:3}
 \LL_X(\omega|_N)&=\eta(X)\omega \\
\LL_X\eta&=d(\eta(X)).\nonumber
\end{align}

Let $N^\omega$ be the leaf space of the associated foliation and assume that $N^\omega$ is a manifold. It follows from equation \eqref{eq:3} that on charts of $N^\omega$ there is a well define conformal symplectic structure and thus that $N^\omega$ is a conformal symplectic manifold. This conformal symplectic manifold is called the \textit{conformal symplectic reduction} of $N$.

\begin{ex}
  Let $Q,Q'$ be two manifolds and $\beta,\beta'$ closed one forms on $Q,Q'$. And let $T^*_{\beta\oplus\beta'}(Q\times Q')\simeq T^*_\beta Q\times T_{\beta'}Q'$ be the conformal symplectic manifold associated to $\pi^*\beta\oplus(\pi')^*\beta'$. Then $T^*_\beta Q\times Q'_0$ is cositropic and its reduction gives back $T^*_\beta Q$.
\end{ex}

Consider a Lagrangian submanifold $L$ of $(M, \eta, \omega)$ and assume the coisotropic $N$ intersects $L$ transversely. Let $\p$ be the projection $N\rightarrow N^\omega$. Then the manifold $L_N=\p(L\cap N)$ is an immersed Lagrangian submanifold of $N^\omega$. (This is linear algebra and thus equivalent to the symplectic case and follows from\cite[Section 5.1]{Bates_Wein} for instance).

Now given a map $F:Q\times Q'\rightarrow \mathbb{R}$, we can take the graph to get a Lagrangian section $d_{\beta\oplus\beta'}F$ of $T^*_{\beta \oplus \beta'}(Q \x Q')$, and apply symplectic reduction as above to the coisotropic $T^*_\beta Q \x Q'_0$ to obtain an immersed exact Lagrangian submanifold of $T_{\beta}Q$. When $Q'=\mathbb{R}^m$ (and $\beta'=0$) and $F$ is quadratic at infinity we denote the corresponding Lagrangian $L_F$. $F$ is called a \textit{$\beta$-generating family} for $L_F$.

Note that $L_F$ lifts to a Legendrian submanifold $\Lambda_F$ of $\mathcal{J}_\eta^1(Q)$ and through the contactomorphism $\phi$ defined above we get a Legendrian submanifold of $\mathcal{J}^1(Q)$, and $F$ is a generating family (in the standard sense) for $\phi(\Lambda_F)$. Chekanov persistence's Theorem \cite{Chekanov_Quasifunctions_Generating} therefore implies that if $L$ in $T_\beta^*Q$ admits a generating family and $L_t$ is an isotopy of $L$ through exact Lagrangians then $L_1$ admits a $\beta$-generating family as well.

We remark that if $L$ is given by a $\beta$-generating family $F:Q\times \R^{m}$ then zeros of $d_\eta F$ (called $\eta$-critical points of $F$) corresponds to intersection points between $L$ and the zero section. All together, this implies

\begin{thm}\label{stablecrit}
Let $\beta$ be a closed $1$-form on a closed manifold $Q$, and define $\op{Stab}_{\beta}(Q) = \min\limits_F \#\{q \in Q | d_\beta F(q)=0\}$, where $F$ is taken among all functions $F:Q\times \R^m\rightarrow \mathbb{R}$ which are quadratic at infinity, for all $m \in \mathbb{N}$. That is, $\op{Stab}_\beta(Q)$ is the stable $\beta$-critical number of $Q$. Let $L \sse T^*_\beta Q$ be a Lagrangian which is Hamiltonian isotopic to the zero section $Z$. Then the number of intersections between $L$ and $Z$ is at least $\op{Stab}_{\beta}(Q)$.
\end{thm}

Form the previous theorem we find that in order to have a lower bounds on the number of intersection points of a deformation a the $0$-section in a cotangent bundle we need to estimate the number of $\beta$-critical points of a generating family $F$. To get such estimates we study the Morse-Novikov homology of a deformation of the function $\ln |F|$.

\section{Overview of Morse-Novikov homology}
\label{sec:overv-morse-novik}

In this section we present the basics of the construction of the Morse-Novikov complex associated to a closed $1$-forms on $Q$. We also recall its essential properties we need to prove Theorem \ref{thm:rigidprinc}. We refer the reader to more comprehensive references like \cite{Farber_oneforms} and \cite{Pajitnov_circledvalued} for more details.

We assume that $Q$ is connected. A closed $1$-forms $\eta$ is \textit{Morse} if near each of its zeros it is the derivative of a non-degenerate quadratic form $h$. It follows from Morse's Lemma that any $1$-forms whose graph intersect the $0$-section transversely is Morse, therefore closed $1$-forms are generically Morse. We refer to a zero of a closed $1$-form $\eta$ as a \textit{critical point} of $\eta$. Given a critical point $q$ of $\eta$ the number of negative eigenvalues of the quadratic form $h$ such that near $q$, $\eta=dh$ is independent of the coordinate system, this number is called the \textit{index} of $q$ written $I_\eta(q)$. 

Given a metric $g$ on $Q$ we define the vector field $\nabla\eta$ to be the dual of $\eta$ with respect to $g$. We denote by $x_t^\eta$ the flow of $\nabla\eta$. We will always assume that there is a compact subset $K$ such that 
\begin{itemize}
\item The frontier of $K$ is a manifold with corners.
\item All critical points of $\eta$ are in $K$ (in particular there are finitely many critical points),
\item $\nabla\eta$ is complete, so $x_t^\eta$ is defined for all $t \in \R$,
\item For every component $V$ of $Q\setminus K$ either for all $q\in V$ and $t\geq 0$ $x_t^\eta(q)\not\in K$ or for all $q\in V$ and $t\leq 0$ $x_t^\eta(q)\not\in K$. 
\end{itemize}
Any critical points $q$ have stable and unstable manifolds $W^s(q)$ and $W^u(q)$ defined by
$W^s(q)=\{q'|\lim_{t\rightarrow \infty}x_t^\eta(q')=q\}$ and $W^u(q)=\{q'|\lim_{t\rightarrow -\infty}x_t^\eta(q')=q\}$.

Note that $W^u(q)$ is a disk of dimension $I_\eta(q)$ and $W^s(q)$ is a disk of dimension $n-I_\eta(q)$. We say that the pair $(\eta,g)$ is \textit{Morse-Smale} if for any critical points $q,q'$ of $\eta$ the disks $W^u(q)$ and $W^s(q')$ intersects transversely. In this situation $W^s(q)\cap W^u(q')$ are open manifolds of dimension $I_\eta(q)-I_\eta(q')$ with an action of $\mathbb{R}$ on them which is free (unless $q=q'$).

For a ring $R$ we denote by $\Lambda_\eta(Q,R)$ the completion of the group ring $R[\pi_1(Q)]$ given by series of the form $\sum_ia_ig_i$ where $\lim_{i\rightarrow\infty}\int_{g_i}\eta=-\infty$.
 
Now for any critical points $q$ of $\eta$ we choose 
\begin{enumerate}
\item A path $g_q$ from $q$ to the base point.
\item An orientation of the tangent space of $q$ (we do not assume $Q$ is orientable).
\end{enumerate}

Given two critical points $q,q'$ of $\eta$ such that $I_\eta(q)-I_\eta(q')=1$ then any component $\gamma$ of $W^s(q)\cup W^u(q')$ is copy of $\mathbb{R}$ which compactifies to a path from $q$ to $q'$ which, together with the capping paths $g_q$ and $g_{q'}$, gives an element $g_\gamma$ of $\pi_1(Q)$.

This gives the series $u_{q,q'}=\sum_\gamma\pm g_\gamma$, which is in $\Lambda_\eta(Q,R)$ since we follow negative trajectories of $\nabla\eta$. The sign $\pm$ is determined by whether or not the chosen orientation of $T_qQ$ transports to the one of $T_{q'}Q$ along the path $\gamma$.

The Morse-Novikov complex of $(\eta,g)$ is given by $C_k(\eta,R)=\oplus_{I_\eta(q)=k}\Lambda(Q,R)\langle q\rangle$ with differential $d(q)=\sum_{I_\eta(q')=k-1}u_{q,q'}q'$. We have that $d^2=0$ and the homology of $(C_*,d)$ is called the Morse-Novikov homology of $\eta$, denoted $H^{\nov}_*(\eta)$.

When $Q$ is compact it follows from a Theorem of Sikorav in \cite{Sik_Thesis} (see also \cite{Poz_Thesis}) that this homology is the homology of $Q$ with local coefficients in $\Lambda_\eta(Q, R)$. Notice that if $Q$ is compact with boundary and if $\nabla \eta$ points inward the boundary then the same results holds.

When $R$ is a principal ring, $\Lambda_\eta (Q,R)$ is a principal ring as well and we call the $k-th$ \textit{Novikov-Betti number} $b^k_\eta$ of $\eta$ the rank of the free part of $H^{\nov}_k(\eta)$. For compact oriented closed manifold they satisfy $b^k_\eta=b^{n-k}_\eta$ (see \cite[Corollary 2.9]{Pajitnov_circledvalued} and \cite[Section 1.5.3]{Farber_oneforms}), for non-oriented on the same is true if $R=\mathbb{Z}_2$. 

Note that $\Lambda_\epsilon(Q,R)$ is a module over $R[\pi_1(Q)]$, the Morse-Novikov homology of $\eta$ is isomorphic to the homology with local coefficient in $H_*(K,\partial_- K;\Lambda_\epsilon(K,R))$ (see \cite{Latour}) where $\partial_- K$ (resp $\partial_+ K$) denotes the points where $\nabla \eta$ points outward (resp. inward) $K$.

 \section{Rigidity of Lagrangian intersection}
\label{sec:rigid-lagr-inters}
We are now ready to prove Theorem \ref{thm:rigidprinc}.

\begin{proof}[Proof of Theorem \ref{thm:rigidprinc}]
From Theorem \ref{stablecrit} it remains to find a lower bound of the number of zeroes of $dF-F\beta$ where $F$ is a function on $Q\times\mathbb{R}^m$, so that there is a compact ball $B \sse \R^m$ such that outside of $Q\times B$, $F(x,\xi)=q(\xi)$ where $q$ is a non-degenerate quadratic form of index $k$. Because we can stabilize $F$ we can assume that $k>0$ without losing generality. Up to a small perturbation not affecting the zeroes we can also assume that $\beta$-critical values of $F$ are not $0$. Fix a metric $g$ on $Q$ inducing the product metric $g\oplus g_0$ on $Q\times \mathbb{R}^m$, we denote by $\nabla: T^*(M\times \mathbb{R}^m)\rightarrow T(M\times \mathbb{R}^m)$ the map induced by this metric.

Choose $\varepsilon>0$ such that for all $(x,\xi)$ such that $|F(x,\xi)|<2\varepsilon$ we have $d_\beta F(x,\xi)\not=0$ and $dF(x,\xi)\not=0$. Let $X_+:=\{(x,\xi)|F(x,\xi)\geq\varepsilon\}$, $X_-:=\{(x,\xi)|F(x,\xi)\leq-\varepsilon\}$ and $X_0:=\{(x,\xi)|-\varepsilon\leq F(x,\xi)\leq \varepsilon\}$.

Let $G$ be a function on $Q\times \mathbb{R}^N$ so that $G(x,\xi)=|F(x,\xi)|$ outside of $X_0$, depends only on $\xi$ outside of $Q\times B$ and $G(x, \xi) > 0$ everywhere. Note that outside of $Q\times B$ the gradient vector field of $d_\beta|F|$ is transverse to $\partial X_\pm$ so we can ensure that $\nabla(dG - G\beta) \pitchfork \dd X_0$. 

We then let $H := \log G$, and notice that the zeros of $\gamma = dH - \beta$ are the same as the zeros of $G(dH - \beta) = dG - G\beta$. Therefore, the zeros of $dF - F\beta$ are the same as the zeros of $\gamma$ on $X_+ \cup X_-$, though $\gamma$ will also have additional zeros in $X_0$. Since $\gamma$ is closed and $[\gamma] = [\beta]$ we can now relate this to $H^\nov_*(Q, [\beta])$.

At first, we work with $H^\nov_*$ as homology with local coefficients in $\Lambda_\beta(Q, \F)$, rather than Morse-Novikov homology. Throughout the next section, $H_*^{\nov}$ will always be taken in the homology class $[\beta]$. Let $r = \op{rk}(H_*^\nov(Q, [\beta]))$, and consider the exact sequence of the pair $(X_0, X_0 \sm Q \x B)$.

$$
\xymatrix{H_*^\nov(X_0, X_0 \sm Q \x B) \ar[r]^{[-1]} &H_*^\nov(X_0 \sm Q \x B) \ar[d]\\
&H_*^\nov(X_0) \ar[ul]&}  
$$

Outside of $B$, we have that $F(x, \xi) = q(\xi)$, and therefore $X_0 \sm Q \x B \cong Q \x S^{k-1} \x S^{m-k-1} \x (-\e, \e) \x (0, \infty)$. Thus 
$$\op{rk}(H_*^\nov(X_0 \sm Q \x B)) = \op{rk}(H^{\nov}_*(Q) \otimes H_*(S^{k-1}\x S^{m-k-1})) = 4r,$$
 and it follows that either $H_*^\nov(X_0, X_0 \sm Q \x B)$ or $H_*^\nov(X_0)$ must have total rank at least $2r$. We consider each case separately, with the goal of proving the inequality 
 \begin{equation}\label{eq:1}
 \op{rk}(H_*^\nov(X_+\dd X_+)) + \op{rk}(H_*^\nov(X_-, \dd X_-)) \geq r.
 \end{equation}

\emph{Case (1) \quad $\op{rk}(H_*^\nov(X_0)) \geq 2r$:}

Consider the Mayer-Vietoris sequence for $(X_+ \cup X_0) \cup (X_- \cup X_0)$, which is

$$
\xymatrix{H_*^\nov(Q \x \R^m) \ar[r]^{[-1]} & H_*^\nov(X_0) \ar[d]\\
&H_*^\nov(X_+ \cup X_0) \oplus H_*^\nov(X_- \cup X_0)\ar[ul]&}  
$$

This implies that $$\op{rk}(H_*^\nov(X_+ \cup X_0)) + \op{rk}(H_*^\nov(X_- \cup X_0)) \leq \op{rk}(H_*^\nov(X_0)) + r.$$ 

By looking at the exact sequence of the pair $(X_+ \cup X_0, X_0)$, we get

$$\op{rk}(H_*^\nov(X_+ \cup X_0, X_0)) \geq \op{rk}(H_*^\nov(X_0)) - \op{rk}(H_*^\nov(X_+ \cup X_0)),$$

and we get a similar inequality for the pair $(X_- \cup X_0, X_0)$. Rearranging these three inequalities gives

\begin{align*}
\op{rk}(H_*^\nov&(X_+ \cup X_0, X_0)) + \op{rk}(H_*^\nov(X_+ \cup X_0, X_0)) \geq \\
& 2 \op{rk}(H_*^\nov(X_0)) - \op{rk}(H_*^\nov(X_+ \cup X_0)) - \op{rk}(H_*^\nov(X_- \cup X_0)) \geq \\
& \op{rk}(H_*^\nov(X_0)) - r \geq r,
\end{align*}

which proves the inequality \ref{eq:1} after applying excision.

\emph{Case (2) \quad $H_*^\nov(X_0, X_0 \sm Q \x B) \geq 2r$:}

Consider the Mayer-Vietoris sequence:

\begin{equation}\label{eq:2}
\xymatrix{H_*^\nov(Q \x \R^m, Q \x \R^m \sm Q \x B) \ar[r]^{[-1]} & H_*^\nov(X_0, X_0 \sm Q \x B) \ar[d]\\
& \hspace{-2.2 in} H_*^\nov(X_+ \cup X_0, (X_+ \cup X_0)\sm Q \x B) \oplus H_*^\nov(X_- \cup X_0, (X_- \cup X_0)\sm Q \x B)\ar[ul]&}  
\end{equation}

Notice that the pair $(X_+ \cup X_0, (X_+ \cup X_0) \sm Q \x B)$ is homeomorphic to $(X_+, X_+ \sm Q \x B)$, since $\nabla F$ presents $X_0$ as a collar neighborhood of $\dd X_+ \cup X_0$. Using excision we have that $$H_*^\nov(X_+, X_+\sm Q \x B) \cong H_*^\nov(X_+ \cap Q \x B, X_+ \cap \dd (Q \x B)).$$ Furthermore since we are using field coefficients, and using the hypothesis of either orientability of $Q$ or $\F = \Z_2$ we can apply Poincar\'e duality to obtain $$H_*^\nov(X_+\cap Q \x B, X_+ \cap \dd (Q \x B)) \cong H_{m+n-*}^\nov(X_+ \cap Q \x B, (\dd X_+)\cap Q \x B).$$

We also note that the pair $(X_+ \cap Q\x B, (\dd X_+) \cap Q\x B)$ is homeomorphic to $(X_+, \dd X_+)$, since $F$ is standard outside $Q \x B$. In summation we obtain an isomorphism

$$H_*^\nov(X_+ \cup X_0, (X_+ \cup X_0)\sm Q \x B) \cong H_{m+n-*}^\nov(X_+, \dd X_+),$$

and also a similar isomorphism for $X_-$. Using excision and K\"unneth we also see that $$H_*^\nov(Q \x \R^m, Q \x \R^m \sm Q \x B) \cong H^\nov_*(Q) \otimes H_*( \R^m, \R^m \sm B) \cong H^\nov_{*+m}(Q).$$ 

Putting these ingredients together, the exact sequence \ref{eq:2} becomes

$$
\xymatrix{H_{*+m}^\nov(Q) \ar[r]^{[-1]} & H_*^\nov(X_0, X_0 \sm Q \x B) \ar[d]\\
&  H_{m+n-*}^\nov(X_+, \dd X_+) \oplus H_{m+n-*}^\nov(X_-, \dd X_-)\ar[ul]&}
$$

This immediately implies the inequality \ref{eq:1}.

We now return to Morse-Novikov homology. By our construction of $\gamma = dH - \beta$, we have that $-\nabla \gamma$ is transverse to $\dd X_+$, and pointing out of $X_+$. Furthermore $-\nabla \gamma$ is transverse to $X_+ \cap \dd (Q \x B)$, pointing into $X_+ \cap \dd (Q \x B)$. Therefore the Morse-Novikov homology of $\gamma$ on the manifold $X_+ \cap Q \x B$ is computing the homology $H^\nov(X_+ \cap Q \x B, (\dd X_+) \cap Q \x B, [\beta])$, which is isomorphic to $H^\nov_*(X_+, \dd X_+, [\beta])$. Similarly the Morse-Novikov homology of $\gamma$ on the manifold $X_- \cap Q\x B$ computes $H^\nov_*(X_-, \dd X_-, [\beta])$. Inequality \ref{eq:1} therefore implies that $\gamma$ must have at least $r$ critical points on $X_+ \cup X_-$, which completes the proof.
\end{proof}

\bibliographystyle{plain}
 \bibliography{Bibliographie_en.bib}
\end{document}